\documentclass[10pt]{article}
\usepackage{amsfonts,amssymb,amsbsy,amsmath,amsthm,mathrsfs,enumerate,verbatim}
 \usepackage[colorlinks=true]{hyperref}
 \hypersetup{urlcolor=blue, citecolor=red}
\newtheorem{theorem}{Theorem}[section]
\newtheorem{proposition}[theorem]{Proposition}
\newtheorem{lemma}[theorem]{Lemma}
\newtheorem{follow}[theorem]{Corollary}
\newtheorem{assumption}[theorem]{Assumption}

\theoremstyle{definition}
\newtheorem{remark}[theorem]{Remark}

\newcommand{\bel}{\begin{equation} \label}
\newcommand{\ee}{\end{equation}}

\newcommand{\pd}{\partial}

\newcommand{\Dom}{{\textit{D}}}
\newcommand{\Ran}{{\textrm{Ran}}}

\newcommand{\R}{{\mathbb R}}
\newcommand{\N}{{\mathbb N}}

\newcommand{\B}{{\mathcal B}}

\newcommand{\F}{{\mathcal F}}
\newcommand{\HH}{{\mathcal H}}
\newcommand{\OO}{{\mathcal O}}
\newcommand{\PP}{{\mathcal P}}

\newcommand{\eps}{{\varepsilon}}
\def\beq{\begin{equation}}
\def\eeq{\end{equation}}
\newcommand{\bea}{\begin{eqnarray}}
\newcommand{\eea}{\end{eqnarray}}
\newcommand{\beas}{\begin{eqnarray*}}
\newcommand{\eeas}{\end{eqnarray*}}
\newcommand{\Pre}[1]{\ensuremath{\mathrm{Re} \left( #1 \right)}}
\newcommand{\Pim}[1]{\ensuremath{\mathrm{Im} \left( #1 \right)}}
{

\begin{document}

\begin{center}
{\Large \bf H\"older stable determination of a quantum scalar potential in unbounded cylindrical domains}

\medskip

\end{center}

\medskip

\begin{center}
{\sc \footnote{CPT, CNRS UMR 7332, Universit\'e d'Aix-Marseille, 13288 Marseille, France \& Universit\'e du Sud-Toulon-Var, 83957 La Garde, France.}{Yavar Kian}, \footnote{Institute of Mathematics, Faculty of Mathematics and Computer Science, Jagiellonian University, 30-348 Krak\'ow, Poland.}{Quang Sang Phan}, \footnote{CPT, CNRS UMR 7332, Universit\'e d'Aix-Marseille, 13288 Marseille, France \& Universit\'e du Sud-Toulon-Var, 83957 La Garde, France.}{Eric Soccorsi}}
\end{center}

\begin{abstract}
We consider the inverse problem of determining the time independent scalar potential of the dynamic Schr\"odinger equation in an infinite cylindrical domain from one boundary Neumann observation of the solution. We prove H\"older stability by choosing the Dirichlet boundary condition suitably.
\end{abstract}

\medskip



\tableofcontents


\section{Statement of the problem and results}
\label{sec-intro} 
\setcounter{equation}{0}

We continue our analysis of the inverse problem of determining the scalar potential $q : \Omega \rightarrow \R$ in a unbounded quantum cylindrical domain $\Omega=\omega \times \R$, where $\omega$ is a
connected bounded open subset of $\R^{n-1}$, $n\geq 2$, with no less than $C^{2}$-boundary $\partial \omega$, from partial Neumann data. This may be equivalently reformulated as to whether the electrostatic disorder occuring in $\Omega$, modelling an idealized straight carbon nanotube, can be retrieved from the partial boundary observation of the quantum wave propagating in $\Omega$. We refer to \cite{KPS1}[\S 1.2] for the discussion on the physical motivations and the relevance of this model. Namely we seek stability in the identification of $q$ from partial Neumann measurement of the solution $u$ to the following initial boundary value problem 
\bel{eq1}
\left\{ \begin{array}{rcll} -i u'-\Delta u + q u & = & 0, & {\rm in}\ Q:=(0,T) \times \Omega\\ u(0,x) &= & u_0(x), & x\in \Omega\\
u(t,x)& = & g(t,x),& (t,x)\in \Sigma:=(0,T) \times \Gamma.\end{array} \right.
\ee 
Here $T>0$ is fixed, $\Gamma:=\pd \omega \times \R$ and the ' stands for $\frac{\partial}{\partial t }$. Since $\Gamma$ is unbounded we make the boundary condition in the last line of \eqref{eq1} more explicit. Writing $x:=(x',x_n)$ with $x':=(x_1,\ldots,x_{n-1}) \in \omega$ for every $x \in \Omega$ we extend the mapping
\beas   
C_0^\infty ((0,T)\times \R; {\rm H}^2(\omega ))& \longrightarrow & {\rm L}^2((0,T) \times \R; H^{3/2}(\partial \omega ))) \nonumber \\
v & \mapsto & [ (t,x_n) \in (0,T) \times \R \mapsto v(t,\cdot,x_n)_{|\partial \omega}],
\eeas
to a bounded operator from $\rm{L}^2 ((0,T)\times \R; {\rm H}^2(\omega ))$ into $\rm{L}^2((0,T)\times \R; H^{3/2}(\partial \omega ))$, denoted by $\gamma_0$. Then
for every $u \in C^0([0,T];{\rm H}^2(\Omega))$ the above mentioned boundary condition reads $\gamma_0 u =g$. 

Throughout the entire paper we choose
\bel{co0}
g:=\gamma_0 G_0,\ \textrm{with}\ G_0(t,x):=u_0(x)+it(\Delta-q_0)u_0(x),\ (t,x) \in Q,
\ee
where $q_0=q_0(x)$ is a given scalar function we shall make precise below.

In the particular case where $q$ is {\it a priori} known outside some given compact subset of $\Omega$
and on the boundary $\Gamma$, it is shown in \cite{KPS1} that the scalar potential may be Lipschitz stably retrieved from one partial Neumann observation of the solution to \eqref{eq1} for suitable initial and boundary conditions 
$u_0$ and $g$. This result is similar to \cite{BP}[Theorem 1], which was derived by Baudouin and Puel for the same operator but acting in a bounded domain.
The main technical assumption common to \cite{BP, KPS1} is that 
\bel{ras}
u \in C^1([0,T];L^\infty(\Omega)).
\ee

In this paper we pursue two main goals. First we want to analyse the direct problem associated to \eqref{eq1}-\eqref{co0} in order to exhibit sufficient conditions on $q$ and $u_0$ ensuring \eqref{ras}. Second, we aim to weaken the compactness condition imposed in \cite{KPS1} on the support of the unknown part of $q$, in the inverse problem of determining the scalar potential appearing in \eqref{eq1} from one partial Neumann observation of $u$.

The following result solves the direct problem associated to \eqref{eq1}-\eqref{co0}. Here and in the remaining part of this text we note $\| w \|_{j,\OO}$, $j \in \N$, for the usual $H^{j}$-norm of $w$ in any subset $\OO$ of $\R^m$, $m \in \N^*$, where $H^0(\OO)$ stands for $L^2(\OO)$.

\begin{theorem}
\label{thm-reg}
Let $k \geq 2$, assume that $\partial \omega$ is $C^{2(k+1)}$, and pick
$$ (q_0,u_0) \in \left(  W^{2k,\infty}(\Omega) \cap C^{2(k-1)}(\overline{\Omega};\R) \right) \times H^{2(k+1)}(\Omega), $$
such that
\bel{co1}
{(-\Delta+q_0)^{2+j}u_0} = 0\ \mbox{on}\ \partial \Omega\ \mbox{for\ all}\ j \in \N_{k-2}:=\{0,1,\ldots,k-2\}.
\ee
Then for each $q \in W^{2k,\infty}(\Omega) \cap C^{2(k-1)}(\overline{\Omega})$ obeying the condition
\bel{co2}
\partial_x^m q = \partial_x^m q_0\ \mbox{on}\ \partial \Omega\ \mbox{for\ all}\ m:=(m_j)_{j=1}^n \in \N^n\ \mbox{with}\ | m |:=\sum_{j=1}^n m_j \leq 2(k-2),
\ee
there is a unique solution $u \in \cap_{j=0}^k C^j([0,T];H^{2(k-j)}(\Omega))$ to the boundary value problem \eqref{eq1}-\eqref{co0}. Moreover, we have the estimate
\bel{eq2}
\sum_{j=0}^k \| u^{(j)} \|_{C^0([0,T];H^{2(k-j)}(\Omega))} \leq C \| u_0 \|_{2(k+1),\Omega},
\ee
where $C>0$ is a constant depending only on $T$, $\omega$, $k$, and $\max(\| q_0 \|_{W^{2k,\infty}(\Omega)}, \| q \|_{W^{2k,\infty}(\Omega)})$.
\end{theorem}

\begin{remark}
It can be seen from the reasonning developped in section \ref{sec-direct} that Theorem \ref{thm-reg} may be generalized to the case of more general boundary conditions than \eqref{co0}. Nevertheless, in order to avoid the inadequate expense of the size of this article, we shall not go into this matter further.
\end{remark}

We now consider the natural number
\bel{defell}
\ell \in \N^* \cap \left( n \slash 4 , n \slash 4 +1 \right].
\ee
Then, applying Theorem \ref{thm-reg} with $k=\ell+1$, we get that
$u \in C^1([0,T]; H^{2 \ell}(\Omega))$ and the estimate $\| u \|_{C^1([0,T];H^{2 \ell}(\Omega))} \leq C \| u_0 \|_{2(\ell+2),\Omega}$. 
Since $2 \ell > n \slash 2$ then $H^{2 \ell}(\Omega)$ is continuously embedded in $L^{\infty}(\Omega)$. This assertion, that is established by Lemma \ref{lm-SET} in section \ref{sec-SET}, extends the corresponding well known Sobolev embedding theorem in $\R^n$ (see e.g. \cite{Br}[Corollary IX.13] or \cite{Ev}[\S 5.10, Problem 18]) to the case of the unbounded cylindrical domain $\Omega$. This immediately entails the:

\begin{follow}
\label{cor-bounded}
Under the conditions of Theorem \ref{thm-reg} for $k=\ell+1$, where $\ell$ is defined by \eqref{defell}, the solution $u$ to \eqref{eq1}-\eqref{co0} satisfies \eqref{ras} and the estimate
$$
\| u \|_{C^1([0,T],L^\infty(\Omega))} \leq C \| u_0 \|_{2(\ell+2),\Omega}.
$$
Here $C>0$ is a constant depending only on $T$, $\omega$ and $\max(\| q_0 \|_{W^{2(\ell+1),\infty}(\Omega)}, \| q \|_{W^{2(\ell+1),\infty}(\Omega)})$.
\end{follow}

For $q_0$ (and $u_0$) as in Theorem \ref{thm-reg}, we aim to retrieve real-valued scalar potentials $q$ verifying
\bel{co3}
| q(x',x_n)-q_0(x',x_n) | \leq a e^{- b \langle x_n \rangle^{d_\epsilon}},\ (x',x_n) \in \Omega,
\ee
where $a>0$, $b>0$, $\epsilon>0$ and $d_\epsilon \in ( 2(1+\epsilon) \slash 3,+\infty)$ are {\it a priori} fixed constants. Here and henceforth the notation $\langle t \rangle $ stands for $(1+t^2)^{1 \slash 2}$, $t \in \R$. Notice that \eqref{co3} weakens the compactness condition imposed in \cite{KPS1}[Theorem 1.1] on the support of the unknown part of $q$. Namely, we introduce the set of admissible potentials as
$$
\mathcal{A}_{\epsilon}(q_0) := \{ q \in W^{2(\ell+1),\infty}(\Omega)\cap \mathcal C^{2 \ell}(\overline{\Omega};\R)\ \mbox{verifying}\ \eqref{co2}\ \mbox{for}\ k=\ell+1\ \mbox{and}\ \eqref{co3} \}.
$$
Our main result on the above mentioned inverse problem is as follows.

\begin{theorem}
\label{thm-inv} 
Let $\pd \omega$, $q_0$ and $u_0$ obey the conditions of Theorem \ref{thm-reg} for $k=\ell+1$, where $\ell$ is the same as in \eqref{defell}. Assume moreover that there are two constants $\upsilon_0>0$ and $\epsilon>0$ such that we have
\bel{co4}
| u_0(x',x_n) |\geq \upsilon_0 \langle x_n \rangle^{-(1+\epsilon) \slash 2},\ (x',x_n) \in \Omega.
\ee
For $M>0$ fixed, we consider two potentials $q_j$, $j=1,2$, in $\mathcal{A}_\epsilon(q_0)$, such that $\| q_j \|_{W^{2(\ell+1),\infty}(\Omega)}\leq M$, and we note $u_j$ the solution to \eqref{eq1}-\eqref{co0} where $q_j$ is substituted for $q$, given by Theorem \ref{thm-reg}. Then, for all $\delta \in (0,b)$, where $b$ is the same as in \eqref{co3}, there exists a
subboundary $\gamma_* \subset \pd \omega$ and a constant $C>0$, depending only on $\omega$, $T$, $M$, $\| u_0 \|_{2 (\ell + 2),\Omega}$, $\delta$, $\epsilon$, $a$, $b$ and $\upsilon_0$, such that the estimate
\bel{stab-ineq}
\| q_1-q_2 \|_{0,\Omega} \leq C \| \partial_\nu u_1'-\partial_\nu u_2' \|_{0, \Sigma_*}^\theta,
\ee
holds for $\Sigma_*:=(0,T) \times \gamma_* \times \R$ and $\theta:=(b-\delta) \slash (2b-\delta)$.
\end{theorem}

It is evident that Theorem \ref{thm-inv} yields uniqueness in the identification of the scalar potential in $\mathcal{A}_\eps(q_0)$ from the knowledge of partial Neumann data for the time-derivative of the solution to \eqref{eq1}-\eqref{co0}:
$$ \forall (q_1,q_2) \in \mathcal{A}_\eps(q_0)^2,\ \left( \partial_\nu u_1'(t,x) = \partial_\nu u_2'(t,x),\ (t,x) \in \Sigma_* \right) \Longrightarrow \left( q_1(x)=q_2(x),\ x \in \Omega \right). $$

Further, it is worth mentioning that Theorem \ref{thm-inv} applies for any subboundary $\Gamma^*=\gamma^* \times \R$ such that
$$ \gamma_* \supset \{x'\in\partial \omega,\ (x'-x_0')\cdot\nu'(x')\geq0\}, $$
for some arbitrary $x_0' \in \R^{n-1} \setminus \overline{\omega}$, where $\nu'(x') \in \R^{n-1}$ denotes the outgoing normal vector to $\partial \omega$ computed at $x'$.
See Assumption \ref{funct-beta} in subsection \ref{sec-Carleman} for more details.

Moreover we stress out that there are actual scalar potentials $q_0$ and initial values $u_0$ fulfilling the conditions of Theorem \ref{thm-inv}. For the sake of completeness a whole class of such matching $q_0$'s and $u_0$'s is indeed exhibited in the concluding remark ending the paper.

The remaining part of the paper is organized with two main sections. Section \ref{sec-direct} contains the proofs of Theorem \ref{thm-reg} and Corollary \ref{cor-bounded}, while Theorem \ref{thm-inv} is shown in section \ref{sec-inv}.


\section{Analysis of the direct problem}
\label{sec-direct}
In this this section we examine the direct problem associated to \eqref{eq1}. To this end we introduce the Dirichlet Laplacian
\bel{dom0} 
A_0:=-\Delta,\ \Dom(A_0):=H_0^1(\Omega) \cap H^2(\Omega), 
\ee
which is selfadjoint in $\HH:=L^2(\Omega)$ (see e.g. \cite{CKS}[Lemma 2.2]), provided $\partial \omega$ is $C^2$.
We perturb $A_0$ by the scalar potential $q \in L^\infty(\Omega)$. Since $q$ is assumed to be real then the operator 
$$ A_q:=A_0+q,\ \Dom(A_q):=\Dom(A_0)=H_0^1(\Omega) \cap H^2(\Omega),$$ 
is selfadjoint in $\HH=L^2(\Omega)$. 

\subsection{Abstract evolution problem}
To study the direct problem associated to \eqref{eq1} we consider the abstract evolution problem
\bel{evo1}
\left\{ \begin{array}{rcll} -i v'(t) + A_q v(t)& = & f(t), & t \in (0,T)\\ v(0) &= & v_0, & 
\end{array}\right.
\ee
with initial data $v_0$ and source $f$. We shall derive smoothness properties of the solution to \eqref{evo1} with the aid of the following technical result, 
which is borrowed from \cite{CKS}[Lemma 2.1]:

\begin{lemma}
\label{lm-evo}
Let $X$ be a Banach space and let $M$ be an m-accretive operator in $X$ with dense domain $\Dom(M)$. Assume that $v_0 \in \Dom(M)$ and $h \in C^1([0,T];X)$.
Then there is a unique solution $v \in C^1([0,T];X) \cap C^0([0,T];\Dom(M))$ to the following evolution problem
\bel{evo2}
\left\{ \begin{array}{rcll} v'(t) +M v(t) & = & h(t), & t \in (0,T)\\ v(0) &= & v_0. & 
\end{array}\right.
\ee
Moreover the estimate
\bel{evo3} 
\| v\|_{C^0([0,T];\Dom(M))} + \| v' \|_{C^0([0,T];X)} \leq C \left( \| v_0 \|_{\Dom(M)} +  \| h \|_{C^1([0,T];X)} \right),
\ee
holds for some positive constant $C$ depending only on $T$.
\end{lemma}

Assume that $\partial \omega$ is $C^2$, let $q \in L^{\infty}(\Omega)$ be real and choose $v_0 \in \Dom(A_q)$ and $f \in C^1([0,T];\HH)$. Since $A_q$ is selfadjoint in $\HH$ then the operator $i A_q$ is $m$-accretive so Lemma \ref{lm-evo} applies with $X=\HH$, $M=iA_q$ and $h=if$: there is a unique solution 
\bel{evo3a}
v \in C^1([0,T];\HH) \cap C^0([0,T];\Dom(A_q))
\ee
to \eqref{evo1}, such that
\bel{evo3b}
\| v \|_{C^0([0,T];\Dom(A_q))} + \| v' \|_{C^0([0,T];\HH)}\leq C_1 \left( \| v_0 \|_{\Dom(A_q)} +\| f \|_{C^1([0,T];\HH)} \right),
\ee
for some constant $C_1>0$ depending only on $T$. 

We shall improve \eqref{evo3a}-\eqref{evo3b} upon assuming higher regularity on and $\partial \omega$, $q$, $v_0$ and $f$.
This requires that the following notations be preliminarily introduced. First, for the sake of definitiness we set $A_q^0 := I$ and $\Dom(A_q^0):=\HH$,
where $I$ denotes the identity operator in $\HH$. Then we put
$$ A_q^k v := A_q (A_q^{k-1} v),\ v \in \Dom(A_q^k):= \{ v \in \Dom(A_q^{k-1}),\ A_q v \in \Dom(A_q^{k-1}) \}, $$
for each $k \in \N^*$, and recall that the linear space $\Dom(A_q^k)$ endowed with the scalar product
$$ \langle v , w \rangle_{\Dom(A_q^k)} := \sum_{j=0}^k \langle A_q^j v , A_q^j w \rangle_{0,\Omega}, $$
is Hilbertian.

\begin{proposition}
\label{pr-reg}
For $k \geq 2$ fixed, assume that $\partial \omega$ is $C^{2(k+1)}$ and let $q \in W^{2(k-1),\infty}(\Omega)$,
$v_0 \in \Dom(A_q^k)$ and $f \in \cap_{j=1}^{k} C^{j}([0,T];\Dom(A_q^{k-j}))$. Then \eqref{evo1} admits a
unique solution 
$$v \in \cap_{j=0}^k C^{j}([0,T];\Dom(A_q^{k-j})),$$
which satisfies the estimate
\bel{evo3c}
\sum_{j=0}^k \| v^{(j)} \|_{C^0([0,T];\Dom(A_q^{k-j}))} \leq C_k \left( \| v_0 \|_{\Dom(A_q^k)} + \| f \|_{C^0([0,T];\Dom(A_q^{k-1}))}  + \sum_{j=1}^{k} \| f^{(j)} \|_{C^0([0,T];\Dom(A_q^{k-j}))} \right),
\ee
where $C_k$ is a positive constant depending only on $T$.
\end{proposition}
\begin{proof}
The proof is by induction on $k \geq 2$. \\
a) Suppose that $k=2$. We consider the space $\tilde{\HH}:=\Dom(A_q)$, which is Hilbertian for the scalar product $\langle v ,w \rangle_{\tilde{\HH}} := \langle v ,w \rangle_{\Dom(A_q)}$, and define the operator $\tilde{A}_q$, acting in $\tilde{\HH}$, by
$$ \tilde{A}_q v := A_q v,\ v \in \Dom(\tilde{A}_q) := \Dom(A_q^2). $$
Since $q$ is real, then the sesquilinear form $a_q[w]:=\|\nabla w \|_\HH^2 + \langle q w , w \rangle_\HH \in \R$ for every $w \in H_0^1(\Omega)$, hence
$$ \Pre{ \langle i \tilde{A}_q v , v \rangle_{\tilde{\HH}} } = \Pre{i (a_q[v] + a_q[A_q v])} = 0,\ v \in \Dom(\tilde{A}_q). $$
Moreover, for all $\varphi \in \tilde{\HH}$ it is clear that the vector $v_\varphi := (i A_q + 1)^{-2} (i \tilde{A_q} +1) \varphi \in \Dom(\tilde{A}_q)$ verifies 
$(i \tilde{A}_q + 1) v_\varphi = \varphi$, proving that $\Ran (i \tilde{A}_q + 1) = \tilde{\HH}$. Thus $i \tilde{A}_q$ is $m$-accretive in $\tilde{\HH}$. Therefore, applying Lemma \ref{lm-evo} for $X=\tilde{\HH}$, $M=i \tilde{A}_q$ and $h=if$, we see that there is a unique function $v \in C^1([0,T];\Dom(A_q)) \cap C^0([0,T];\Dom(A_q^2))$, that is solution to
\bel{evo5}
\left\{ \begin{array}{rcll} -i v'(t) + \tilde{A}_q v(t)& = & f(t), & t \in (0,T)\\ v(0) &= & v_0, & 
\end{array}\right.
\ee
and fulfills
\bel{evo6} 
\| v \|_{C^0([0,T];\Dom(A_q^2))} + \| v' \|_{C^0([0,T];\Dom(A_q))} \leq C_1 \left( \| v_0 \|_{\Dom(A_q^2)} +  \| f \|_{C^1([0,T];\Dom(A_q))} \right).
\ee
Since $v$ is solution to \eqref{evo5}, it is also solution to \eqref{evo1}. Hence $v$ is the solution to \eqref{evo1}, by uniqueness. Further, since $A_q$ is linear bounded from $\Dom(A_q)$ into $\HH$ and $v \in C^1([0,T];\Dom(A_q))$, we have $A_q v \in C^1([0,T];\HH)$ with 
\bel{evo7}
(A_q v)'(t) = A_q v'(t).
\ee
In light of \eqref{evo1} this entails that $v'(t)=i (f(t) - A_q v(t) ) \in C^1([0,T];\HH)$. As a consequence we have $v \in C^2([0,T];\HH)$ and
\bel{evo8}
-i v''(t) + A_q v'(t) = f'(t),\ t \in (0,T),
\ee 
according to \eqref{evo7}. Moreover \eqref{evo8} yields
$$ \| v'' \|_{C^0([0,T];\HH)} \leq  \| A_q v'\|_{C^0([0,T];\HH)} + \| f' \|_{C^0([0,T];\HH)} \leq \| v'\|_{C^0([0,T];\Dom(A_q))} + \| f' \|_{C^0([0,T];\HH)}. $$
From this and \eqref{evo6} then follows that
$$ 
\sum_{j=0}^2 \| v^{(j)} \|_{C^0([0,T];\Dom(A_q^{2-j}))} \leq \max(1,2 C_1) \left( \| v_0 \|_{\Dom(A_q^2)} +  \| f \|_{C^1([0,T];\Dom(A_q))} \right),
$$
which entails \eqref{evo3c} for $k=2$.\\
b) Let us now prove the result for $k \geq 3$. We suppose that the assertion holds true for all nonzero natural number smaller or equal to $k-1$ and choose $v_0 \in \Dom(A_q^k)$ and
$f \in \cap_{j=1}^{k} C^{j}([0,T];\Dom(A_q^{k-j}))$. We know from part a) that the solution $v$ of \eqref{evo1} is lying in $C^{2}([0,T];\HH) \cap C^{1}([0,T];\Dom(A_q))$ and satisfies \eqref{evo8}. Putting $w=v'$ we thus find out that $w \in C^{1}([0,T];\HH) \cap C^{0}([0,T];\Dom(A_q))$ is solution to the problem
$$
\left\{ \begin{array}{rcll} -i w'(t) + A_q w(t)& = & f'(t), & t \in (0,T)\\ w(0) &= & i ( f(0)-A_q v_0). & 
\end{array}\right.
$$
Since $f(0)-A_q v_0 \in \Dom(A_q^{k-1})$ and $f' \in \cap_{j=1}^{k-1} C^{j}([0,T];\Dom(A_q^{k-1-j}))$, we deduce from the induction hypothesis that
$w \in \cap_{j=0}^{k-1} C^j([0,T];\Dom(A_q^{k-1-j}))$
and that $\sum_{j=0}^{k-1} \| w^{(j)} \|_{C^0([0,T];\Dom(A_q^{k-1-j}))}$ is majorized, up to the multiplicative constant $C_{k-1}$, by the following upper bound:
\beas
&  &   \| f(0)- A_q v_0  \|_{\Dom(A_q^{k-1})} + \| f' \|_{C^0([0,T];\Dom(A_q^{k-2}))} + \sum_{j=1}^{k-1} \| f^{(j+1)} \|_{C^0([0,T];\Dom(A_q^{k-1-j}))} \\
& \leq & \| v_0  \|_{\Dom(A_q^{k})} + \| f \|_{C^0([0,T];\Dom(A_q^{k-1}))} + \| f' \|_{C^0([0,T];\Dom(A_q^{k-2}))} + \sum_{j=2}^{k} \| f^{(j)} \|_{C^0([0,T];\Dom(A_q^{k-j}))}.
\eeas
This entails
$$
\sum_{j=0}^{k-1} \| w^{(j)} \|_{C^0([0,T];\Dom(A_q^{k-1-j}))} \leq C_{k-1} \left( \| v_0  \|_{\Dom(A_q^{k})} + \| f \|_{C^0([0,T];\Dom(A_q^{k-1}))} + \sum_{j=1}^{k} \| f^{(j)} \|_{C^0([0,T];\Dom(A_q^{k-j}))} \right).
$$
As a consequence we have
\bel{evo10}
v \in \cap_{j=1}^{k} C^j([0,T];\Dom(A_q^{k-j})),
\ee
and
\bel{evo11}
\sum_{j=1}^{k} \| v^{(j)} \|_{C^0([0,T];\Dom(A_q^{k-j}))}
\leq C_{k-1} \left( \| v_0  \|_{\Dom(A_q^{k})} + \| f \|_{C^0([0,T];\Dom(A_q^{k-1}))} + \sum_{j=1}^{k} \| f^{(j)} \|_{C^0([0,T];\Dom(A_q^{k-j}))} \right).
\ee
Now it remains to show that $v \in C^0([0,T];\Dom(A_q^k))$. This can be seen from \eqref{evo1} and \eqref{evo10}, entailing
$$
A_q v = f + i v' \in C^0([0,T];\Dom(A_q^{k-1})).
$$
Therefore we have $\| v \|_{C^0([0,T];\Dom(A_q^k))} \leq \| f \|_{C^0([0,T];\Dom(A_q^{k-1}))} + \| v' \|_{C^0([0,T];\Dom(A_q^{k-1}))}$, which together with
\eqref{evo11}, yields \eqref{evo3c}.
\end{proof}

\begin{remark}
\label{rm-evo}
In light of \eqref{evo3a}-\eqref{evo3b} it is clear that the statement of Proposition \ref{pr-reg} holds true for $k=1$ as well.
\end{remark}

\subsection{Characterizing $\Dom(A_q^k)$ for $k \in \N^*$}
In light of Proposition \ref{pr-reg} we need to describe more explicitely the domain of the operator $A_q^k$, $k \in \N^*$. Namely, under suitable assumptions on $\partial \omega$ and $q$ we shall characterize $\Dom(A_q^k)$ as a subset of $H^{2k}(\Omega)$. The main tool we use in the derivation of this result, stated in Proposition \ref{pr-dom}, is the elliptic boundary regularity property in $\Omega$ established in Lemma \ref{lm-re}.

\subsubsection{Elliptic boundary regularity in $\Omega$}
We now extend the classical elliptic boundary regularity result for the Dirichlet Laplacian, which is well known in any sufficiently smooth bounded domain of $\R^n$ (see e.g. 
\cite{LM1}[Chap. 2, Theorem 5.1] or \cite{Ev}[\S 6.3, Theorem 5]), to the case of the unbounded cylindrical waveguide $\Omega$. Namely we consider the following problem
\bel{re0}
\left\{ 
\begin{array}{rccl} 
-\Delta v & = & \varphi & \textrm{in}\ \Omega \\ v &= & 0 & \textrm{on}\ \pd \Omega,
\end{array} \right.
\ee
and establish the:
\begin{lemma}
\label{lm-re} 
Let $\ell \in \N$, assume that $\partial \omega$ is $C^{2 (\ell+1)}$ and choose $\varphi \in H^{2 \ell}(\Omega)$. Then \eqref{re0} admits a unique solution $v \in H^{2(\ell+1)}(\Omega)$,
obeying
\bel{re1}
\| v \|_{ H^{2(\ell+1)}(\Omega))}\leq C_\ell \| \varphi \|_{H^{2 \ell}(\Omega)},
\ee
where $C_\ell>0$ is a constant depending only on $\omega$ and $\ell$.
\end{lemma}

\begin{proof}
The proof is by induction on $\ell \in \N$ and relies essentially on the following decomposition of the Dirichlet Laplacian $A_0$
\bel{re2}
\F A_0 \F^{-1} = \int_{\R}^{\oplus} \widehat{A}_{0,p} dp.
\end{equation}
Here $\F$ denotes the partial Fourier transform with respect to $x_n$, i.e.
\[
(\F w)(x',p) = \widehat{w}(x',p)=\frac{1}{(2 \pi)^{1 \slash 2}} \int_{\R} e^{-i p x_n} w(x',x_n) dx_n,\ (x',p) \in \Omega, 
\]
and 
\bel{re2a} 
\widehat{A}_{0,p}:=-\Delta_{x'}+p^2,\ p \in \R,
\ee
is the selfadjoint operator in $L^2(\omega)$ generated by the closed quadratic form 
$$\widehat{a}_{0,p}[w]:=\int_{\omega} (|\nabla_{x'} w(x')|^2 + p^2 |w(x')|^2 )dx',\ w \in \Dom(\widehat{a}_{0,p}):=H_0^1(\omega). $$
Since $\omega$ is a bounded domain with boundary no less than $C^2$, we have 
\bel{re2b}
\Dom(\widehat{A}_{0,p})=H_0^1(\omega) \cap H^2(\omega),\ p \in \R,
\ee
from \cite{Ag}. For further reference we notice for all $p \in \R$ that 
$$  \langle \widehat{A}_{0,p} w , w \rangle_{0,\omega} = \widehat{a}_{0,p}[w] \geq (c_0(\omega) + p^2) \| w \|_{0,\omega},\ w \in H_0^1(\Omega) \cap H^2(\Omega), $$
where $c_0(\omega)>0$ is the constant appearing in the Poincar\'e inequality associated to the bounded domain $\omega$. As a consequence the operator $\widehat{A}_{0,p}$ is boundedly invertible in $L^2(\omega)$ and it holds true that
\bel{re3}
\| \widehat{A}_{0,p}^{-1} \|_{\B(L^2(\omega))} \leq (c_0(\omega)+p^2)^{-1},\ p \in \R. 
\ee
In light of \eqref{dom0} and \eqref{re2b} we see that $v \in H_0^1(\Omega) \cap H^2(\Omega)$
is solution to \eqref{re0} iff $\widehat{v}(p) \in H_0^1(\omega) \cap H^2(\omega)$ verifies 
\bel{re4}
\left\{ 
\begin{array}{rlcl} 
\widehat{A}_{0,p} \widehat{v}(p) & = & \widehat{\varphi}(p) & \textrm{in}\ \omega \\ \widehat{v}(p) &= & 0 & \textrm{on}\ \pd \omega,
\end{array} \right.
\ee
for a.e. $p \in \R$. \\
a) Having said that we first examine the case $\ell=0$. We pick $\varphi \in L^2(\Omega)$ and assume that $\partial \omega$ is $C^2$. Since $\widehat{\varphi}(p) \in L^2(\omega)$ for a.e. $p \in \R$ we deduce from \cite{Ev}[\S 6.3, Theorem 5] that there is a unique solution $\widehat{v}(p) \in H^2(\omega)$ to \eqref{re4}. Taking into account that
$$ - \Delta_{x'} \widehat{v}(p) = g(p)\ \mbox{with}\ g(p):=\widehat{\varphi}(p)  -p^2 \widehat{v}(p) \in L^2(\omega),\ p \in \R, $$
we get (see e.g. the remark following \cite{Ev}[\S 6.3,  Theorem 5]):
\bel{re5}
\| \widehat{v}(p) \|_{2,\omega} \leq C(\omega) \| g(p) \|_{0,\omega} \leq C(\omega) ( \| \widehat{\varphi}(p) \|_{0,\omega} + p^2 \| \widehat{v}(p) \|_{0,\omega} ),\ p \in \R.
\ee
Here and henceforth $C(\omega)$ denotes some generic positive constant depending only on $\omega$.
On the other hand, since $\widehat{v}(p)=\widehat{A}_{0,p}^{-1} \widehat{\varphi}(p)$ by \eqref{re4}, we have
$$ \| \widehat{v}(p) \|_{0,\omega} \leq ( c_0(\omega) + p^2 )^{-1} \| \widehat{\varphi}(p) \|_{0,\omega},\ p \in \R, $$
according to \eqref{re3}. From this and \eqref{re5} then follows that
\bel{re6}
\| \widehat{v}(p) \|_{2,\omega} \leq C(\omega) \| \widehat{\varphi}(p) \|_{0,\omega},\ p \in \R.
\ee
Since $p \mapsto \widehat{\varphi}(p) \in L^2(\R;L^2(\omega))$ as $\varphi \in L^2(\Omega)$, then we have $v = \F^{-1} \widehat{v} \in L^2(\R;H^2(\omega))$ and
\bel{re7}
\| v \|_{L^2(\R;H^2(\omega))} \leq C(\omega) \| \varphi \|_{0,\Omega},
\ee
by \eqref{re6}. Further, taking into account that
$$ \| \widehat{A}_{0,p} \widehat{v}(p) \|_{0,\omega}^2 = \| \Delta_{x'} \widehat{v}(p) \|_{0,\omega}^2 + 2 p^2 \| \nabla_{x'} \widehat{v}(p) \|_{0,\omega}^2 + p^4 \| \widehat{v}(p) \|_{0,\omega}^2 = \| \widehat{\varphi}(p) \|_{0,\omega}^2,\ p \in \R, $$
from \eqref{re2a} and \eqref{re4}, we get that $\int_{\R} p^4 \| \widehat{v}(p) \|_{0,\omega}^2 dp \leq \| \varphi \|_{0,\Omega}^2$. This entails that $v \in H^2(\R;L^2(\omega))$ obeys
\bel{re8}
\| v \|_{H^2(\R;L^2(\omega))} \leq \| \varphi \|_{0,\Omega}.
\ee
Applying \cite{LM1}[Chap. 1, Theorem 7.4] to $v \in L^2(\R;H^2(\omega)) \cap H^2(\R;L^2(\omega))$ we find that $v \in H^2(\Omega)$. Moreover, the norm in $L^2(\R;H^2(\omega)) \cap H^2(\R;L^2(\omega))$ being equivalent to the usual one in $H^2(\Omega)$, \eqref{re7}-\eqref{re8} yields
$$ \| v \|_{2,\Omega} \leq C(\omega) \| \varphi \|_{0,\Omega}. $$
b) For $\ell \geq 1$ fixed, assume that $\partial{\omega}$ is $C^{2(\ell +1)}$ and let $\varphi \in H^{2 \ell}(\Omega)$. We suppose in addition that the claim obtained by substituting $k \in \N_{\ell-1}$ for $\ell$ in Lemma \ref{lm-re}, and the following estimate
\bel{re9}
\| \widehat{v}(p) \|_{2(k+1),\omega} \leq C(\omega) \sum_{j=0}^{k} p^{2j} \| \widehat{\varphi}(p) \|_{2(k-j),\omega},\ p \in \R,
\ee
both hold true. Notice that \eqref{re9} actually reduces to \eqref{re6} in the particular case where $\ell-1=0$.
Taking into account that $\widehat{\varphi}(p) \in H^{2 \ell}(\omega)$ for a.e. $p \in \R$ and $\partial{\omega}$ is $C^{2(\ell +1)}$, the solution $\widehat{v}(p)$ to
\eqref{re4} is uniquely defined in $H^{2(\ell+1)}(\omega)$ by  \cite{Ev}[\S 6.3, Theorem 5], and it satisfies:
\bel{re10}
\| \widehat{v}(p)\|_{2(\ell + 1),\omega} \leq C(\omega) ( \| \widehat{\varphi}(p) \|_{2 \ell, \omega} + p^2 \| \widehat{v}(p)\|_{2 \ell,\omega} ).
\ee
Moreover, due to \eqref{re9}, the estimate $p^2 \| \widehat{v}(p)\|_{2 \ell,\omega} \leq C(\omega) \sum_{j=1}^{\ell} p^{2j} \| \widehat{\varphi}(p) \|_{2(\ell-j),\omega}$ holds for a.e. $p \in \R$, hence
\bel{re11}
\| \widehat{v}(p)\|_{2(\ell + 1),\omega} \leq C(\omega) \sum_{j=0}^{\ell} p^{2j} \| \widehat{\varphi}(p) \|_{2(\ell-j),\omega},\ p \in \R,
\ee
from \eqref{re10}.
As a consequence we have $v \in L^2(\R;H^{2(\ell +1)}(\omega))$ and the following estimate:
\bel{re12}
\| v \|_{L^2(\R;H^{2(\ell +1)}(\omega))} \leq C(\omega) \sum_{j=0}^\ell \| \varphi \|_{H^{2j}(\R;H^{2(\ell-j)}(\omega))} \leq C(\omega) \| \varphi \|_{2 \ell,\Omega}.
\ee
Further, multiplying \eqref{re4} by $p^{2 \ell}$ yields
$p^{2(\ell +1)} \widehat{v}(p) = p^{2 \ell} ( \widehat{\varphi}(p) + \Delta_{x'}\widehat{v}(p))$ for a.e. $p \in \R$, so we get
$$ p^{2(\ell + 1)} \| \widehat{v}(p) \|_{0,\omega} \leq p^{2 \ell} ( \| \widehat{\varphi}(p) \|_{0,\omega} + \| \widehat{v}(p) \|_{2,\omega} ) \leq C(w) p^{2 \ell} \| \widehat{\varphi}(p) \|_{0,\omega},\ p \in \R, $$
by applying \eqref{re9} with $k=0$. From this then follows that $v \in H^{2(\ell+1)}(\R;L^2(\omega))$, with
\bel{re13}
\| v \|_{H^{2(\ell+1)}(\R;L^2(\omega))} \leq C(\omega) \| \varphi \|_{H^{2\ell}(\R;L^2(\omega))}\leq C(\omega) \| \varphi \|_{H^{2\ell}(\Omega)}.
\ee
Summing up, we have $v \in L^2(\R;H^{2(\ell+1)}(\omega)) \cap H^{2(\ell+1)}(\R;L^2(\omega))$. Interpolating with \cite{LM1}[Chap. 1, Theorem 7.4], we thus find out that $v \in H^{2(\ell+1)}(\Omega)$. Finally,
\eqref{re1} follows from \eqref{re12}-\eqref{re13} and the equivalence of norms in $L^2(\R;H^{2(\ell+1)}(\omega)) \cap H^{2(\ell+1)}(\R;L^2(\omega))$ and $H^{2(\ell+1)}(\Omega)$. This and \eqref{re11} proves the assertion of the lemma.
\end{proof}

\subsubsection{More on $\Dom(A_q^k)$, $k \in \N^*$}

Armed with Lemma \ref{lm-re} we are now in position to prove the following:
\begin{proposition}
\label{pr-dom}
Let $k \in \N^*$, $q \in W^{2(k-1),\infty}(\Omega)$, and assume that $\partial \omega$ is $C^{2(k+1)}$. Then we have
\bel{dom1}
\Dom(A_q^k) = \{ v \in H^{2k}(\Omega),\ A_q^j v \in H_0^1(\Omega)\ \mbox{for\ all}\ j \in \N_{k-1} \}.
\ee
Moreover $\| \cdot \|_{\Dom(A_q^k)}$ is equivalent to the usual norm in $H^{2k}(\Omega)$. Namely, if $\| q \|_{W^{2(k-1),\infty}(\Omega)} \leq M$ for some $M>0$, then we may find a constant $c_k=c_k(M,\omega)>0$, such that we have
\bel{dom2}
\frac{1}{c_k} \| v \|_{\Dom(A_q^k)} \leq  \| v \|_{2k,\Omega} \leq c_k \| v \|_{\Dom(A_q^k)},\ v \in \Dom(A_q^k).
\ee
\end{proposition}
\begin{proof}
We proceed by induction on $k \in \N^*$.\\
a) Case $k=1$. Recalling that $\Dom(A_q)=H_0^1(\Omega) \cap H^2(\Omega)$, it is enough to prove the right inequality in \eqref{dom2}.  This can be done by noticing that $v \in \Dom(A_q)$ is solution to
\eqref{re0} with $\varphi = A_q v - q v \in \HH$, which entails
$$\| \varphi \|_{0, \Omega} \leq M \| v \|_{0,\Omega} + \| A_q v \|_{0,\Omega} \leq \max(1,M) \| v \|_{\Dom(A_q)}, $$
and then applying \eqref{re1} with $\ell=0$.\\
b) Let $k \geq 2$ be fixed and suppose that \eqref{dom1}-\eqref{dom2} is valid for all $j \in \N_{k-1}$. Assume moreover that $q \in W^{2(k-1),\infty}(\Omega)$ and $\partial \omega$ is $C^{2(k+1)}$. Pick $v \in \Dom(A_q^{k})$ and put $\varphi= A_q v - q v$. Both $v$ and $A_q v$ are in $\Dom(A_q^{k-1})$, which is embedded in $H^{2(k-1)}(\Omega)$ by induction assumption, so we have $\varphi \in H^{2(k-1)}(\Omega)$. Therefore, $v$ being solution to \eqref{re0}, it holds true that $v \in H^{2k}(\Omega)$, from Lemma \ref{lm-re}. Moreover the r.h.s. of \eqref{dom2}, and hence both sides, since the l.h.s. is completely straightforward, follows by applying \eqref{re1} with $\ell=k-1$. Further, the induction hypothesis combined with the fact that $v \in \Dom(A_q^{k-1})$ (resp., $A_q v \in \Dom(A_q^{k-1})$ for all $j \in \N_{k-1}^*$) yields $v \in H_0^1(\Omega)$ (resp., $A_q^j v \in H_0^1(\Omega)$ for all $j \in \N_{k-1}^*$). Summing up, we have
\bel{dom3}
\Dom(A_q^{k}) \subset \{ v \in H^{2k}(\Omega),\ A_q^j v \in H_0^1(\Omega)\ \mbox{for\ all}\ j \in \N_{k-1} \}.
\ee
On the other hand, each $v \in H^{2k}(\Omega)$ such that $A_q^j v \in H_0^1(\Omega)$ for all $j \in \N_{k-1}$, fulfills
$$
A_q^m v \in H^{2(k-1)}(\Omega)\ \mbox{and}\ A_q^j (A_q^m v) \in H_0^1(\Omega),\ j \in \N_{k-2},\ m=0,1,
$$
so we get that $v \in \Dom (A_q^{k-1})$ (resp., $A_q v \in \Dom (A_q^{k-1})$), by applying the induction assumption to $A_q^m v$ with $m=0$ (resp., $m=1$). As a consequence we have 
$v \in \Dom(A_q^{k})$, showing that the inclusion relation may be reversed in \eqref{dom3}, which proves the desired result.
\end{proof}

This immediately entails the:
\begin{follow}
\label{cor-dom}
For $k \geq 2$, assume that $\partial \omega$ is $C^{2(k+1)}$ and let $q_0, q \in W^{2(k-1),\infty}(\Omega) \cap C^{2(k-2)}(\overline{\Omega})$ fulfill
\eqref{co2}.
Then we have $\Dom(A_q^k) = \Dom(A_{q_0}^k)$.
\end{follow}

\subsection{Proof of Theorem \ref{thm-reg}}
Evidently, $u$ is solution to \eqref{eq1}-\eqref{co0} iff $v:=u-G_0$ is solution to the following boundary value problem
\bel{reg1}
\left\{ \begin{array}{rcll} -i v'-\Delta v + q v & = & f  & {\rm in}\ Q\\ v(0,x) &= & 0, & x \in \Omega\\
v(t,x)& = & 0, & (t,x) \in \Sigma, \end{array} \right.
\ee 
where
\bel{reg2}
f:=i G_0' + (\Delta-q) G_0 = it (-\Delta + q_0)^2 u_0 - (q-q_0) G_0,
\ee
and $G_0$ is the function defined by \eqref{co0}. We first prove that
\bel{reg3}
f \in C^{\infty}([0,T];\Dom(A_q^{k-1})),
\ee
and
\bel{reg3b}
\| f  \|_{C^1([0,T];\Dom(A_q^{k-1}))} \leq C \| u_0 \|_{2(k+1),\Omega},
\ee
where $C$ denotes some generic positive constant depending only on $T$, $\omega$ and $M$.
To do that we notice from \eqref{co0} that $G_0 \in C^{\infty}([0,T];H^{2k}(\Omega))$, with
\bel{reg3c}
\| G_0 \|_{C^1([0,T];H^{2k}(\Omega))} \leq C \| u_0 \|_{2(k+1),\Omega}.
\ee
As a consequence we have $(q-q_0) G_0 \in C^{\infty}([0,T];H^{2(k-1)}(\Omega))$. Hence Proposition \ref{pr-dom} yields
\bel{reg4}
(q-q_0) G_0 \in C^{\infty}([0,T];\Dom(A_q^{k-1})),
\ee
since every $(-\Delta+q_0)^j (q-q_0) G_0$, $j \in \N_{k-2}$, vanishes on $\partial \Omega$, by \eqref{co2}. Here we used the identity
$\Dom(A_q^{k-1})=\Dom (A_{q_0}^{k-1})$, arising from \eqref{co2} and Corollary \ref{cor-dom}. Similarly, since $(-\Delta + q_0)^2 u_0 \in H^{2(k-1)}(\Omega)$, we deduce from
\eqref{co1} that $(-\Delta + q_0)^2 u_0 \in \Dom(A_q^{k-1})$. This, \eqref{reg2} and \eqref{reg4}, entails \eqref{reg3}, while \eqref{reg3b} follows from \eqref{reg3c} and the basic estimate $\| (-\Delta + q_0)^2 u_0 \|_{2(k-1),\Omega} \leq C \| u_0 \|_{2(k+1),\Omega}$.

Further, refering to \eqref{reg1}, we see that $v$ is solution to \eqref{evo1}. Since $f \in C^{\infty}([0,T];\Dom(A_q^{k-1}))$, by \eqref{reg3}, then $v$ is uniquely defined in $\cap_{j=0}^k C^j([0,T];\Dom(A_q^{k-j}))$ from Proposition \ref{pr-reg}. Moreover, bearing in mind that $f$ is affine in $t$, we get
\bel{reg5}
\sum_{j=0}^k \| v^{(j)} \|_{C^0([0,T];\Dom(A_q^{k-j}))} \leq C \| u_0 \|_{2(k+1),\Omega},
\ee
by \eqref{evo3c} and \eqref{reg3b}.
Now, upon recalling that $G_0$ is an affine function of $t$, the claim of Theorem \ref{thm-reg} follows readily from the identity $u=v+G_0$, \eqref{reg3c}, \eqref{reg5} and Proposition \ref{pr-dom}.

\subsection{A Sobolev embedding theorem in $\Omega$}
\label{sec-SET}
The derivation of Corollary \ref{cor-bounded} from Theorem \ref{thm-reg} boils down to the following result.

\begin{lemma}
\label{lm-SET} 
Let $k \in \N^*$ satisfy $k > n \slash 2$ and assume that $\partial \omega$ is $\mathcal C^k$. Then we have $H^k(\Omega) \subset L^\infty(\Omega)$. Moreover there is a constant $c>0$, depending only on $n$, $k$ and $\omega$, such that the estimate
\bel{sv0}
\| h \|_{L^\infty(\Omega)}\leq c \| h \|_{k,\Omega},
\ee
holds for all $h \in H^k(\Omega)$.
\end{lemma}
\begin{proof} 
Since $\omega$ is a bounded domain of $\R^{n-1}$ with $C^k$ boundary, there exists an extension operator 
\bel{sv1}
P \in \B(H^j(\omega); H^j(\R^{n-1})),\ j \in \N_k,
\ee 
such that
\bel{sv2} 
(Pf)_{\vert\omega}=f, f \in L^2(\omega),
\ee
according to \cite{LM1}[Chap. 1, Theorem 8.1]. Next, for all $f \in L^2(\Omega)$, put
$$
\PP f(x',x_n)=[Pf(\cdot,x_n)](x'),\ (x',x_n) \in \Omega.
$$
It is apparent from \eqref{sv1} that $\PP \in \B(H^m(\R;H^j(\omega)); H^m(\R;H^j(\R^{n-1}))$ for all natural numbers $m$ and $j$ such that $m+j\leq k$. As a consequence we have
\bel{sv3}
\PP \in \B(H^k(\Omega); H^k(\R^{n})).
\ee 
Moreover, it follows readily from \eqref{sv2} that
\bel{sv4}
(\PP h)_{\vert\Omega}=h,\ h\in L^2(\Omega).
\ee
Pick $h \in H^k(\Omega)$. Since $\PP h \in H^k(\R^n)$ by \eqref{sv3}, then the Sobolev embedding theorem \cite{Br}[Corollary IX.13] yields $\PP h \in L^\infty(\R^n)$ and the estimate
\bel{sv5}
\| \PP h \|_{L^\infty(\R^n)} \leq C \| \PP h \|_{k,\R^n},
\ee
where the constant $C>0$ is independent of $h$. From this and \eqref{sv4} then follows that $h \in L^\infty(\Omega)$, with
\bel{sv6}
\| h \|_{L^\infty(\Omega)} \leq \| \PP h \|_{L^\infty(\R^n)}.
\ee
Putting \eqref{sv3} and \eqref{sv5}-\eqref{sv6} together we end up getting \eqref{sv0}.
\end{proof}


\section{Stability estimate}
\label{sec-inv}
In this section we prove \eqref{stab-ineq} by adapting the Bukhgeim-Klibanov method introduced in \cite{BK}. It is by means of a Carleman estimate specifically designed for the system under consideration.

\subsection{Linearization and time symmetrization}
\label{sec-linearization}

Set $\rho:=q_1-q_2$ so that $u:=u_1-u_2$ is solution to the boundary value problem
\bel{ip1}
\left\{
\begin{array}{rcll}
-i u'-\Delta u + q_1 u & = &  - \rho u_2 & \mbox{in}\ Q\\  u(0,x) & = & 0, & x \in\Omega\\ u(t,x) & = & 0, &(t,x) \in \Sigma.
\end{array}
\right.
\ee
Since $u \in C^2([0,T];H^{2 (\ell -1)}(\Omega)) \cap C^1([0,T];H^{2 \ell}(\Omega))$ by Theorem \ref{thm-reg}, we may differentiate \eqref{ip1} w.r.t. $t$, getting
\bel{ip2}
\left\{
\begin{array}{rcll} 
-i v'-\Delta v + q_1 v &= & -\rho u_2' & \mbox{in}\ Q \\  v(0,x)  &= & -i \rho u_0, & x\in\Omega \\ v(t,x) & = & 0, & (t,x) \in \Sigma,
\end{array}
\right.
\ee
where $v:=u'$. Since $q_j \in \mathcal{A}_\epsilon(q_0)$, $j=1,2$, we have $\rho u_0 \in H_0^1(\Omega) \cap H^2(\Omega)$. Consequently $v \in C^1([0,T]; L^2(\Omega)) \cap C^0([0,T];H_0^1(\Omega) \cap H^2(\Omega))$ by applying Remark \ref{rm-evo} and Proposition \ref{pr-dom} for $k=1$. Further, putting $u_2(-t,x)=\overline{u_2(t,x)}$ for all $(t,x) \in [-T,0) \times \Omega$ and bearing in mind that $u_0$ and $q_j$, $j=1,2$, are real-valued, we deduce from \eqref{ip2} that the function $v$, extended on $[-T,0) \times \Omega$ by setting $v(t,x):=-\overline{v(-t,x)}$, is the $C^1([-T,T];L^2(\Omega)) \cap C^0([-T,T];H_0^1(\Omega) \cap H^2(\Omega))$-solution to the system
\bel{ip3}
\left\{
\begin{array}{rcll} 
-i v'-\Delta v + q_1 v &= & -\rho u_2' & \mbox{in}\ \tilde{Q}:=(-T,T) \times \Omega \\  v(0,x)  & = & -i \rho u_0, & x\in\Omega \\ v(t,x) & = & 0, & (t,x) \in \tilde{\Sigma}:=(-T,T) \times \Gamma.
\end{array}
\right.
\ee
The main tool needed for the derivation of \eqref{stab-ineq} is a global Carleman inequality for the Schr\"odinger equation in \eqref{ip3}. We use the estimate derived in \cite{KPS1}[Proposition 3.3], that is specifically designed for unbounded cylindrical domains of the type of $\Omega$.

\subsection{Global Carleman estimate for the Schr\"odinger equation in $\Omega$}
\label{sec-Carleman}
Given the Schr\"odinger operator acting in $(C_0^{\infty})'(\tilde{Q})$,
\bel{H} 
L := -  i \partial_t - \Delta,
\ee
we introduce a function
$\tilde{\beta} \in C^4(\overline{\omega};\R_+)$ and an open subset $\gamma_*$ of $\pd \omega$, satisfying the following conditions:
\begin{assumption}
\label{funct-beta}
\hfill \break \vspace*{-.5cm}
\begin{enumerate}[(i)]
\item $\exists C_0>0$ such that the estimate $|\nabla_{x'} \tilde{\beta}(x')| \geq C_0$ holds for all $x' \in \omega$;
\item ${\partial}_{\nu} {\tilde{\beta}}(x') := \nabla_{x'} {\tilde{\beta}}(x'). \nu (x') < 0$ for all $x' \in
    \partial \omega \backslash \gamma_*$, where $\nu$ is the outward unit normal vector to $\partial \omega$;
\item $\exists \Lambda_1>0$, $\exists \epsilon>0$  such that we have $\lambda |\nabla_{x'}  \tilde{\beta}(x') \cdot
    \zeta|^2 + D^2 \tilde{\beta} (x',\zeta, \zeta) \geq \epsilon  |\zeta|^2$ for all $\zeta \in  \R^{n-1}$, $x' \in \omega$ and
    $\lambda  > \Lambda_1$, where $D^2 \tilde{\beta}(x'):=\left( \frac{\partial^2 \tilde{\beta}(x')}{\partial x_i \partial x_j} \right)_{1 \leq i,j \leq n-1}$ and $D^2 \tilde{\beta} (x',\zeta, \zeta)$ denotes the $\R^{n-1}$-scalar product of $D^2 \tilde{\beta}(x') \zeta$ with $\zeta$.
\end{enumerate}
\end{assumption}
Notice that there are actual functions $\tilde{\beta}$ verifying Assumption \ref{funct-beta}, such as $\omega \ni x' \mapsto | x'- x_0' |^2$, for an arbitrary $x_0' \in \R^{n-1} \setminus \overline{\omega}$ and a subboundary $\gamma_* \supset \{x' \in \partial \omega,\ (x'-x_0') \cdot \nu'(x') \geq0 \}$.

Next, for all $x=(x',x_n) \in \Omega$, put
\bel{defbeta} 
\beta(x):= \widetilde{\beta}(x')+K,\ {\rm where}\ K:= r \|\tilde{\beta}\|_{\infty}\ {\rm for\
some}\ r>1, 
\ee 
and define the two following weight functions for $\lambda>0$: 
\bel{defphieta}
\varphi(t,x):=\frac{e^{\lambda  \beta(x)}}{(T+t)(T-t)}\ {\rm and}\ \eta(t,x):=\frac{e^{2\lambda K} -
e^{\lambda \beta(x)}}{(T+t)(T-t)},\ (t,x) \in \tilde{Q}. 
\ee 
Finally, for all $s>0$, we introduce the two following
operators acting in $(C_0^{\infty})'(\tilde{Q})$:
\bel{M1} 
M_1 : = i \partial_t +
\Delta  + s^2 |\nabla \eta |^2\ {\rm and}\ M_2: = i s \eta' + 2 s \nabla \eta \cdot \nabla  + s (\Delta \eta). 
\ee
It is apparent that $M_1$ (resp. $M_2$) is the adjoint (resp. skew-adjoint) part of the operator $e^{-s \eta} L e^{s \eta}$,
where $L$ is given by \eqref{H}. 

We may now state the following global Carleman estimate, that is borrowed from \cite{KPS1}[Proposition 3.3].

\begin{proposition}
\label{pr-carleman} 
Let $\tilde{\beta}$ and $\gamma_*$ obey Assumption \ref{funct-beta}, let $\beta$, $\varphi$ and $\eta$ be given by \eqref{defbeta}-\eqref{defphieta}, and let the operators $M_j$, 
$j=1,2$, be defined by \eqref{M1}.
Then there are two constants $s_0>0$ and $C>0$, depending only on $T$, $\omega$ and $\gamma_*$,
such that the estimate
\beas
& & s  \|  {\rm e}^{-s \eta}  \nabla_{x'} w   \|_{0,\tilde{Q}}^2
+s^3  \| e^{-s \eta} w  \|_{0,\tilde{Q}}^2  + \sum_{j=1,2} \|   M_j  e^{-s \eta} w \|_{0,\tilde{Q}}^2 \nonumber  \\
& \leq  & C  \left(  s \|  e^{-s \eta} \varphi^{1/2}  ( \partial_{\nu} \beta)^{1/2}
\partial_{\nu}  w   \|_{0,\tilde{\Sigma}_*}^2+  \|   e^{-s \eta}   L  w  \|_{0,\tilde{Q}}^2 \right),
\eeas
holds for all $s \geq s_0$ and any function $w \in L^2(-T,T;  {\rm H}^1_0( \Omega ) ) $ verifying $ L w \in
L^2(\tilde{Q})$ and $\partial_{\nu} w \in L^2(-T,T;L^2(\Gamma_*))$. Here $\Gamma_*$ (resp., $\tilde{\Sigma}_*$) stands for $\gamma_* \times \R$ (resp., $(-T,T) \times \gamma_* \times \R$).
\end{proposition}

\subsection{Proof of Theorem \ref{thm-inv}}
We use the same notations as in \S \ref{sec-linearization} and, for the sake of notational simplicity, we denote the various positive constants appearing in the derivation of Theorem \ref{thm-inv} by $C$. Following \cite{KPS1}[Lemmae 3.3 \& 3.4], we start by establishing the coming technical result with the aid of Proposition \ref{pr-carleman}.
\begin{lemma}
\label{lm-inv}
For all $s>0$, we have the estimate:
$$
\| e^{-s\eta(0,\cdot)} \rho u_0 \|_{0,\Omega}^2 \leq C \left( s^{-3 \slash 2} \| e^{-s\eta(0,\cdot)} \rho u_2' \|_{0,Q}^2 + s^{-1 \slash 2} \| e^{-s\eta(0,\cdot)} \partial_\nu v \|_{0,\Sigma_*}^2 \right).
$$
\end{lemma} 
\begin{proof}
Put $\phi:=e^{-s \eta} v$. In light of \eqref{defbeta}-\eqref{defphieta} it holds true that $\lim\limits_{\substack{t \downarrow (-T) }} \eta(t, x)= +\infty$ for all $x \in \Omega$, hence 
$\lim\limits_{\substack{t \downarrow (-T) }} \phi(t, x)= 0$.
As a consequence we have
\bel{ip4}
\| \phi (0,\cdot) \|_{0,\Omega}^2 =  \int_{(-T,0) \times \Omega} (| \phi |^2)'(t,x)  dt dx = 2 \Pre{\int_{(-T,0) \times \Omega} \phi'(t,x) \overline{\phi(t,x)} dt dx} . 
\ee 
On the other hand, \eqref{M1} and the Green formula yield
$$ \Pim {\int_{(-T,0) \times \Omega} M_1 \phi(t,x)  \overline{\phi(t,x) } dt dx} = \Pre {\int_{(-T,0) \times \Omega}   \phi'(t,x)   \overline{\phi(t,x) } dt dx }  +  W, $$
with
$W:= \Pim {
\int_{(-T,0) \times \Omega} \Delta \phi(t,x)  \overline{\phi(t,x) }dt dx + s^2 \| (\nabla \eta) \phi  \|_{0, (-T,0) \times \Omega}^2}= \Pim{\| \nabla \phi \|_{0, (-T,0) \times \Omega}^2}=0$.
This, along with \eqref{ip4} and the identity $\| \phi(0,\cdot) \|_{0,\Omega} = \|  e^{-s \eta(0, \cdot )} v(0,\cdot ) \|_{0,\Omega}$ entail
\beas 
& & \|  e^{-s \eta(0, \cdot )} v(0,\cdot ) \|_{0,\Omega}^2 = 2  \Pim { \int_{(-T,0) \times  \Omega} M_1 \phi(t,x)  \overline{\phi(t,x) }  dt dx} \\
& \leq &  2 \| M_1 \phi  \|_{0,\tilde{Q}} \|  \phi  \|_{0,\tilde{Q}} \leq s^{-3 \slash 2}  \left(  s^3 \| {\rm e}^{-s \eta} v \|_{0,\tilde{Q}}^2 +  \|  M_1  e^{-s \eta} v \|_{0,\tilde{Q}}^2 \right).
\eeas
Finally, the desired result follows from this upon recalling \eqref{ip3} and applying Proposition \ref{pr-carleman} to $v$.
\end{proof}
Fix $y>0$. In virtue of Lemma \ref{lm-inv}, it holds true that
$$
\| e^{-s\eta(0,\cdot)} \rho u_0 \|_{0,\omega \times (-y,y)}^2 \leq C \left(s^{-3 \slash 2} \| e^{-s\eta(0,\cdot)} \rho u_2' \|_{0,Q}^2 + s^{-1 \slash 2} \mu
 \right),\ s>0,
$$
where $\mu:=\| \partial_\nu u_1'-\partial_\nu u_2' \|_{0,\Sigma_*}^2$. This entails
\bel{ip5}
( \upsilon_0^2 \langle y \rangle ^{-(1+\epsilon)}-C s^{-3 \slash 2} ) \| e^{-s \eta(0,\cdot)} \rho \|_{0,\omega \times (-y,y)}^2 \leq 
C\left( s^{-3 \slash 2} \|  \rho \|_{0,\omega \times (\R \setminus (-y,y))}^2 + s^{-1 \slash 2} \mu \right),\ s>0,
\ee
since $|u_0(x) |\geq \upsilon_0 \langle y \rangle ^{-(1+\epsilon) \slash 2}$ for all $x \in \omega \times (-y,y)$, 
by \eqref{co4}, $\| u_2' \|_{L^{\infty}(Q)} \leq C$ by Corollary \ref{cor-bounded}, and $\eta(0,x) \geq 0$ for all $x \in \Omega$.
Taking $s=(\upsilon_0^2 \slash (2C) )^{-2 \slash 3} \langle y \rangle ^{2(1+\epsilon) \slash 3}$ in \eqref{ip5} and bearing in mind that $\| \eta(0,.) \|_{L^\infty(\Omega)} \leq e^{2K} \slash T^2$, we thus obtain that
\bel{ip6}
\| \rho \|_{0,\omega \times (-y,y)}^2 \leq C  e^{C \langle y \rangle ^{2(1+\epsilon) \slash 3}}\left( \| \rho \|_{0,\omega \times (\R \setminus (-y,y))}^2 + \langle y \rangle^{2(1+\epsilon) \slash 3} \mu \right).
\ee
Moreover we know from \eqref{co3} that
$$
\|  \rho \|_{0,\omega \times (\R \setminus (-y,y))}^2  \leq 4 a^2 | \omega|^2 \int_{\R \setminus (-y,y)} e^{-2b \langle x_n \rangle^{d_\epsilon}} dx_n \leq \left( 4a^2  | \omega|^2 \int_\R e^{-\delta \langle x_n \rangle^{d_\epsilon}} dx_n \right) e^{-(2 b-\delta) \langle y \rangle^{d_\epsilon}}.
$$
From this and \eqref{ip6} then follows that 
\bel{ip7}
\| \rho \|_{0,\omega \times (-y,y)}^2 
\leq C  \langle y \rangle ^{2(1+\epsilon) \slash 3} e^{C \langle y \rangle ^{2(1+\epsilon) \slash 3}}  \left( e^{-(2b-\delta) \langle y \rangle ^{d_\epsilon}} +\mu \right).
\ee
Set $\mu_\delta:=e^{-(2b-\delta)}$. We examine the two cases $\mu \in  (0,\mu_\delta)$ and $\mu \geq \mu_\delta$ separately. If $\mu \in (0,\mu_\delta)$ we take
$y=y(\mu):=\left( \left(-\frac{\ln \mu}{2b-\delta}\right)^{2 \slash d_\epsilon}-1\right)^{1 \slash 2}$
in \eqref{ip7}, getting:
$$
\| \rho \|_{0,\omega \times (-y,y)}^2  \leq C  \langle y \rangle^{2(1+\epsilon) \slash 3} e^{C \langle y \rangle ^{2(1+\epsilon) \slash 3}-(2b-\delta) \langle y \rangle ^{d_\epsilon}}.
$$
Since $d_\epsilon>2(1+\epsilon) \slash 3$, this entails that
\bel{ip8}
\| \rho \|_{0,\omega \times (-y,y)}^2  \leq C \left( \sup_{t>1} t^{2(1+\epsilon) \slash 3} e^{Ct^{2(1+\epsilon) \slash 3}-\delta t^{d_\epsilon}} \right) e^{-2(b-\delta) \langle y \rangle ^{d_\epsilon}} \leq  C \mu^{2 \theta},\ \mu \in (0,\mu_\delta).
\ee
On the other hand, it follows from \eqref{co3} that
\bel{ip9}
\| \rho \|_{0,\omega \times (\R \setminus (-y,y))}^2 \leq C \left( \int_\R e^{-2\delta  \langle x_n \rangle} dx_n \right)  e^{-2(b-\delta) \langle y \rangle ^{d_\epsilon}} \leq  C \mu^{2 \theta},\  \mu \in (0,\mu_\delta).
\ee
Putting \eqref{ip8}-\eqref{ip9} together, we obtain that
\bel{ip10}
\| \rho \|_{0,\Omega}^2 \leq C \mu^{2 \theta},\   \mu \in (0,\mu_\delta).
\ee
Finally, if $\mu \geq \mu_\delta$, we use the estimate $\| \rho \|_{0,\Omega}^2 \leq 
4 a^2 |\omega | (\int_\R e^{-2b \langle x_n \rangle^{d_\epsilon}}dx_n)$, arising from \eqref{co3}, and find:
$$
\| \rho \|_{0,\Omega}^2  \leq \left( \frac{4 a^2 |\omega | \int_\R e^{-2b \langle x_n \rangle^{d_\epsilon}}dx_n}{\mu_\delta^{2 \theta}} \right) \mu^{2 \theta} \leq C \mu^{2 \theta},\ \mu \geq \mu_\delta.
$$
Therefore, \eqref{stab-ineq} follows from this and \eqref{ip10}.

\subsection{Concluding remark}
In this subsection we build a class of scalar potentials $q_0$ and initial conditions $u_0$, fulfilling the conditions of Theorem \ref{thm-inv}. To do that we first introduce
two functions $u_b, q_b: \R\to\R$, defined by
\bel{ex0}
u_b(y):=c \langle y \rangle^{-(1+\epsilon) \slash 2},\ q_b(y):=\frac{u_b''(y)}{u_b(y)},\ y\in\R,
\ee
where $c>0$ is some fixed constant. Evidently, we have 
\bel{ex1}
u_b \in H^{2(\ell+2)}(\R)\ \mbox{and}\ q_b \in W^{2(\ell+1),\infty}(\R) \cap C^{2 \ell}(\R).
\ee
Since $\omega$ is an open bounded subset of $\R^{n-1}$ then $\pd \omega$ is compact in $\R^{n-1}$ so we may find $\mathcal{O} \subset \R^{n-1}$, open and bounded, that is a neighborhood of $\pd \omega$. Further, let $\chi \in C^\infty_0(\R^{n-1})$ verify $\chi(x')=1$ for $x' \in \mathcal{O}$ and $0 \leq \chi(x') \leq 1$ for all $x' \in \R^{n-1}$, pick $q_i \in W^{2(\ell+1),\infty}(\Omega)\cap C^{2 \ell}(\overline{\Omega})$ and choose $u_i \in H^{2(\ell+2)}(\Omega)$ such that
\bel{ex2}
u_i(x',x_n) \geq u_b(x_n)=c \langle x_n \rangle^{-(1+\epsilon) \slash 2},\ (x',x_n) \in \Omega.
\ee
Then, upon setting for all $x=(x',x_n) \in \Omega$,
\bel{ex3}
q_0(x):=\chi(x')q_b(x_n)+(1-\chi(x'))q_i(x)\ \mbox{and}\ u_0(x):=\chi(x') u_b(x_n) + (1-\chi(x')) u_i(x),
\ee
it is apparent from \eqref{ex1} that $q_0 \in W^{2(\ell+1),\infty}(\Omega)\cap C^{2 \ell}(\overline{\Omega})$ and $u_0\in H^{2(\ell+2)}(\Omega)$.
Moreover, it follows readily from \eqref{ex0} and \eqref{ex2}-\eqref{ex3} that
$$ u_0(x) \geq \chi(x') u_b(x_n)+(1-\chi(x')) u_b(x_n) \geq c \langle x_n \rangle^{-(1+\epsilon) \slash 2},\ x=(x',x_n) \in\Omega. $$
Finally, to show that $(q_0,u_0)$ satisfies \eqref{co1}, we invoke \eqref{ex3}, getting
$$ u_0(x',x_n)=u_b(x_n)\ \mbox{and}\ q_0(x',x_n)=q_b(x_n),\ x=(x',x_n)\in \mathcal{O} \times\R, $$
and consequently
$$ (-\Delta + q_0)u_0(x',x_n)=-u_b''(x_n)+q_b(x_n) u_b(x_n)=0,\ (x',x_n) \in \Omega \cap (\mathcal{O} \times \R), $$
by \eqref{ex0}. This immediately yields
$$ (-\Delta + q_0)^{2+j} u_0(x) = (-\Delta+q_0)^{1+j} (-\Delta+q_0)u_0(x)=0,\ x \in \Omega \cap (\mathcal{O} \times \R),\ j=0,1,\cdots,\ell-1, $$
which, in turn, entails \eqref{co1}, since $\mathcal{O} \times\R$ is a neighborhood of $\pd \Omega$ in $\R^n$.


\bigskip


\begin{thebibliography} {[10]}
\frenchspacing \baselineskip=12 pt plus 1pt minus 1pt



\bibitem{Ag} {\sc S. Agmon}, {\em Lectures on Elliptic Boundary Value Problems},  Van Nostrand Co., Inc., Princeton, N.J.-Toronto-London, 1965.


\bibitem{BP} {\sc L. Baudouin, J.-P. Puel}, {\em Uniqueness and stability in an inverse problem for the Schr\"odinger equation}, Inverse Probl., {\bf 18} (2002), 1537-1554.





\bibitem{Br}{\sc H. Brezis}, {\em Analyse Fonctionnnelle. Th\'eorie et applications}, Collection math\'ematiques appliqu\'ees pour la ma\^{i}trise, Masson, Paris 1993.

\bibitem{BK} {\sc A. L. Bukhgeim, M. V. Klibanov}, {\em Uniqueness in the large of a class of multidimensional
inverse problems}, Sov. Math. Dokl. {\bf 17} (1981), 244–247.
 




\bibitem{CKS}{\sc M. Choulli, Y. Kian, E. Soccorsi}, {\em Stable determination of time dependent scalar potential from boundary measurements in a periodic quantum waveguide}, arXiv:1306.6601.



\bibitem{Ev}{\sc L. C. Evans}, {\em Partial Differential Equations},  Amer. Math. Soc., Graduate Studies in Mathematics, vol. 19.








\bibitem{KPS1} {\sc Y. Kian, Q. S. Phan, E. Soccorsi}, {\em Carleman estimate for infinite cylindrical quantum domains and application to inverse problems}, arXiv:1305.1042.



\bibitem{LM1} {\sc J.-L. Lions, E. Magenes}, {\em Probl\`emes aux limites non homog\`enes et applications}, vol. 1, Dunod (1968).















\end{thebibliography}
\end{document}